\documentclass[12pt]{article}
\pdfoutput=1

\usepackage{amsmath}
\usepackage[utf8]{inputenc}
\usepackage[UKenglish]{babel}

\usepackage[normalem]{ulem} 
\usepackage{aliascnt}
\usepackage{amssymb}
\usepackage{amsthm}
\usepackage{adjustbox}
\usepackage{changepage}
\usepackage{faktor}
\usepackage{mathtools}
\usepackage{amsfonts}
\usepackage{enumerate}
\usepackage{enumitem}
\usepackage{multicol}
\usepackage{float}
\usepackage{titlefoot}
\usepackage{tikz}
\usetikzlibrary{graphs, positioning}
\usepackage{tikz-cd} 
\usepackage{comment}
\newcommand{\Mod}[1]{\ (\mathrm{mod}\ #1)}
\usepackage[round,comma,sort
]{natbib}
\usepackage{hyperref}
\hypersetup{
    colorlinks=true,       
    linkcolor=blue,          
    citecolor=olive,       
    filecolor=magenta,      
    urlcolor=blue          
}
\usepackage[capitalise, nameinlink]{cleveref}

\newcommand{\Z}{\mathbb{Z}}
\newcommand{\ga}{\Gamma}

\newcommand{\ve}{\varepsilon}

\newcommand{\gp}[2]{\langle #1 \, | \, #2 \rangle}
\newcommand{\sgp}[1]{\langle #1 \rangle}

\DeclareMathOperator{\fl}{FL}\DeclareMathOperator{\llt}{LL}

\DeclareMathOperator{\supp}{Supp}

\DeclareMathOperator{\lk}{Lk}
\DeclareMathOperator{\st}{St}
\DeclareMathOperator{\Star}{St}

\def\coloneqq{\mathrel{\mathop\mathchar"303A}\mkern-1.2mu=}

\newcommand{\myRed}[1]{\textbf{\textcolor{red}{#1}}}
\newcommand{\comm}[1]{}
\begin{document}

\title{Conjugacy problem in T-RAAGs}
\author{Gemma Crowe and Islam Foniqi}

\date{}
\maketitle

\theoremstyle{plain}
\newtheorem{theorem}{Theorem}[section]
\newtheorem{conj}[theorem]{Conjecture}
\newtheorem{notation}[theorem]{Notation}
\newtheorem*{note}{Note}
\theoremstyle{definition}

\setlength{\parindent}{0em}
\setlength{\parskip}{0.5em} 
\author{}

\newaliascnt{conjecture}{theorem}
\newtheorem{conjecture}[conjecture]{Conjecture}
\aliascntresetthe{conjecture}
\providecommand*{\conjectureautorefname}{Conjecture}

\newaliascnt{lemma}{theorem}
\newtheorem{lemma}[lemma]{Lemma}
\aliascntresetthe{lemma}
\providecommand*{\lemmaautorefname}{Lemma}

\newaliascnt{cor}{theorem}
\newtheorem{cor}[cor]{Corollary}
\aliascntresetthe{cor}
\providecommand*{\corautorefname}{Corollary}

\newaliascnt{claim}{theorem}
\newtheorem{claim}[claim]{Claim}

\newaliascnt{notation}{theorem}
\aliascntresetthe{notation}
\providecommand*{\notationautorefname}{Notation}

\aliascntresetthe{claim}
\providecommand*{\claimautorefname}{Claim}

\newaliascnt{remark}{theorem}
\newtheorem{remark}[remark]{Remark}
\aliascntresetthe{remark}
\providecommand*{\remarkautorefname}{Remark}

\newtheorem*{claim*}{Claim}
\theoremstyle{definition}

\newaliascnt{definition}{theorem}
\newtheorem{definition}[definition]{Definition}
\aliascntresetthe{definition}
\providecommand*{\definitionautorefname}{Definition}

\newaliascnt{prop}{theorem}
\newtheorem{prop}[prop]{Proposition}
\aliascntresetthe{prop}
\providecommand*{\propautorefname}{Proposition}

\newaliascnt{example}{theorem}
\newtheorem{example}[example]{Example}
\aliascntresetthe{example}
\providecommand*{\exampleautorefname}{Example}

\newaliascnt{question}{theorem}
\newtheorem{question}[question]{Question}
\aliascntresetthe{question}
\providecommand*{\questionautorefname}{Question}

\newaliascnt{maintheorem}{theorem}
\newtheorem{maintheorem}[maintheorem]{Theorem}
\aliascntresetthe{maintheorem}
\providecommand*{\maintheoremautorefname}{Theorem}
\renewcommand{\themaintheorem}{\Alph{maintheorem}}

\begin{abstract}
In this paper, we construct an implementable algorithm which solves the conjugacy problem in twisted right-angled Artin groups (T-RAAGs). In certain cases, the complexity is known to be linear, by reducing the problem to the twisted conjugacy problem in right-angled Artin groups. We also show that T-RAAGs are biautomatic, providing an alternative solution to the conjugacy problem. \\

\end{abstract}
\unmarkedfntext{\emph{Keywords}: Conjugacy problem, twisted right-angled Artin groups, biautomaticity}

\section{Introduction}\label{Introduction}
In geometric group theory, the class of Artin groups plays a major role, with many long standing questions, such as decidability of the word problem, remaining open in general. One well-studied subclass of Artin groups is the class of \emph{right-angled Artin groups} (RAAGs). Also known as graph groups, RAAGs are defined by a finite simplicial graph, where each edge of the graph represents a commutation relation; see \citep{charney2007introduction} for a survey on RAAGs. The notion of a \emph{twisted Artin group} was introduced in \citep{clancy2010homology}; this class of groups also appears in \citep{pride1986tits} as a subclass of \emph{generalized Artin groups}. This article explores the subclass of \emph{twisted right-angled Artin groups} (T-RAAGs) as a natural extension of RAAGs, where we also allow directed edges in our defining graph. 

Similar to RAAGs, the presentation of T-RAAGs is given by generators and relations: there are finitely many generators and there is at most one relation between any two distinct generators~$a,b$. This relation can be either a commutation~$ab = ba$, or a so-called {\it Klein relation}~$aba = b$. This presentation can be encoded by the use of mixed graphs, where vertices represent the generators of the group, while edges give rise to the relations. An undirected edge, connecting generators~$a$ and~$b$, gives rise to the commutation relation~$ab = ba$, while a directed edge with origin at~$a$ and terminus at~$b$ defines the Klein relation~$aba = b$. If generators~$a,b$ are not connected by an edge, they are free of relations.

As basic and important examples of T-RAAGs, we have both the fundamental groups of a torus:~$\Z^2 = \gp{a, b}{ab = ba}$, and of the Klein bottle:~$K = \gp{a, b}{aba = b}$. Note that~$K$ is not a RAAG, because RAAGs are bi-orderable \citep{Duchamp1992}, but~$K$ is not. We refer to T-RAAGs also with the name \emph{twisted graph groups}. More details regarding their normal forms and other properties can be found in \citep[Chapter~3]{foniqi2022}, and also in \citep{foniqi2024twisted}. 

The focus of this paper is decision problems in T-RAAGs. We recall the three main problems in group theory here:
\begin{itemize}
    \item[$(i)$] \emph{Word problem}: For a group~$G$ generated by a set~$X$, does there exists an algorithm to decide whether~$w \in (X \sqcup X^{-1})^{\ast}$ represents the identity element of~$G$? 
    \item[$(ii)$] \emph{Conjugacy problem}: For a group~$G$ generated by a set~$X$, does there exists an algorithm to decide whether any two words~$u,v \in \left(X \sqcup X^{-1}\right)^{\ast}$ represent conjugate elements in~$G$?
    \item[$(iii)$] \emph{Isomorphism problem}: Given two finite group presentations~$\langle X \mid R \rangle$ and~$\langle Y \mid S \rangle$, does there exists an algorithm to decide whether these presentations represent isomorphic groups?
\end{itemize}
All three of these problems are decidable for RAAGs. Indeed the first two problems are solvable in linear time by \citep*{CGW}, and the third problem is solvable by \citep{Droms1987}. It is natural therefore to ask whether any of these three problems are solvable for T-RAAGs. We note that the word problem for T-RAAGs follows from the normal form theorem given in \citep{foniqi2022, foniqi2024twisted}. The isomorphism problem however remains wide open -- see \citep*{blumer2025droms} for a solution in a certain subclass of `Droms T-RAAGs'.
The following result constitutes the main contribution to this paper, in solving the second of these three decision problems.

\begin{maintheorem}[\cref{thm:CP solvable}]\label{thm:main result intro}
    {\it The conjugacy problem is decidable in T-RAAGs. }
\end{maintheorem}
We prove a conjugacy criterion for T-RAAGs (see \autoref{prop:conj criteria}) that extends the well-known result for RAAGs, which is that any two cyclically reduced elements in a RAAG are conjugate if and only if they are related by a sequence of commutation relations and cyclic permutations. For T-RAAGs, extra care is required to account for Klein relations. For example, if $aba = b$ in our T-RAAG, then $a$ and $a^{-1}$ represent conjugate elements, but are clearly not related by commutation relations or cyclic permutations. Our criterion is then used to construct an implementable solution to check when two elements in a T-RAAG are conjugate. 

An implementable linear time solution for the conjugacy problem in RAAGs was constructed by \citep*{CGW}, using a construction known as \emph{pilings}. This was adapted in \citep{Crowe2024} to solve the \emph{twisted conjugacy problem} for RAAGs in certain cases. This result can be used to prove linear time complexity for the conjugacy problem in T-RAAGs in certain cases.  

\begin{maintheorem}[\cref{thm:subcase result}]\label{thm:subcase result intro}
  {\it  Let $T(\ga)$ be a T-RAAG with generators $V\ga = \{x_{1}, \dots, x_{n}, y\}$, such that the only Klein relations are of the form $x_{i}yx_{i} = y$, where $i \in I$ for some subset $I \subseteq \{1, \dots, n\}$. Then the conjugacy problem in~$T(\Gamma)$ can be solved in linear time. }
\end{maintheorem}
T-RAAGs of this form are cyclic extensions of RAAGs, and so the conjugacy problem for~$T(\ga)$ is equivalent to solving the twisted conjugacy problem in the RAAG~$A(\widehat{\ga})$ -- see \citep*{Orbits}. 

We finish the article by showing T-RAAGs are \emph{biautomatic}, a property which also holds in RAAGs. There is extensive literature on automatic and biautomatic groups, and we refer the reader to \citep{epstein1992word} for background information. Biautomatic groups have solvable conjugacy problem \citep{Gersten1997}, and so the following result provides an alternative proof of \autoref{thm:main result intro}.

\begin{maintheorem}[\cref{thm:biauto}]\label{thm:biauto intro}
    {\it T-RAAGs are biautomatic.} 
\end{maintheorem}

The structure of this paper is as follows. \Cref{Definitions_preliminaries_and_notation} serves to establish the notation and provide the main definitions of the paper; this includes functions that help us keep track of syllable shuffling in words representing group elements of T-RAAGs. This will allow us to establish a conjugacy criterion, similar to that of graph products, which we can then use to prove our main results in \cref{sec:CP in T-RAAGS}. We finish by proving biautomaticity of T-RAAGs in \cref{sec:biauto}.

\section{Preliminaries and notation}\label{Definitions_preliminaries_and_notation}

For a finite set $A$, denote by $A^\ast$ the corresponding free monoid, that is, the set of all words over $A$, including the empty word $\varepsilon$. For a word $w \in A^\ast$, denote by $l(w)$ its length, which is the number of letters used to  write it.

For a subset~$X$ of a group, we let~$X^{\pm} = X \sqcup X^{-1}$, where~$X^{-1} = \{x^{-1} \mid x \in X \}$.
Our groups will be given by finite presentations~$G = \gp{S}{R}$; any element~$g \in G$ can be written as a word over~$S^{\pm}$. 
We let~$\sgp{X}$ denote the subgroup of~$G$ generated by~$X$. 

For a group~$G = \langle X \rangle$, and words~$u,v \in (X^{\pm})^*$, we let~$u = v$ denote equality of words, $u =_{G} v$ denote the equality of group elements represented by~$u$ and~$v$, and $u \sim v$ denote when $u$ and $v$ represent conjugate elements of $G$. For a word~$w \in X^{\ast}$, we denote by~$l(w)$ the word length of~$w$ over~$X$. For a group element~$g$ in~$G = \langle X \rangle$, we define the \emph{length} of~$g$, denoted~$| g |_{X}$, to be the length of a shortest representative word for the element~$g$ over~$X^{\pm}$, i.e.~$|g|_{X} = \min \{ l(w) \mid w\in (X^{\pm})^*, \; w =_G g\}$; if~$X$ is fixed or clear from the context, we write~$|g|$. 

\comm{A word~$w \in (X^{\pm})^*$ is called a \emph{geodesic} if~$l(w) = |\pi(w)|$, where~$\pi \colon (X^{\pm})^* \rightarrow G$ is the natural projection.}
\comm{
\subsection{Right-angled Artin groups}\label{Right-angled Artin groups}
The graphs we use to define graph groups are called \emph{simplicial}, they come with no loops, and no multiple edges.
\begin{definition}\label{def_simplicial_graph}
A {\it simplicial graph~$\Gamma$} is a pair~$\Gamma = (V, E)$, where~$V = V\ga$ is a non-empty set whose elements are called \emph{vertices}, and~$E = E\ga \subseteq \{\{x,y\} \mid x,y \in V, x\neq y\}$ is a set of paired distinct vertices, whose elements are called \emph{edges}.
\end{definition}
For a finite simplicial graph~$\Gamma = (V, E)$ we use~$A(\Gamma)$ to denote the right-angled Artin group based on~$\Gamma$, defined by the presentation 
$$A(\Gamma) \coloneqq \langle\, V \mid uv = vu \text{ whenever } \{u,v\}\in E\,\rangle.$$

\begin{definition}\label{def: link and star}
    For~$v \in V\ga$, we define 
    \[ \lk(v) = \{u \in V\ga \mid \{u,v\} \in E\ga \}
    \]
    For a subset~$A \subseteq V\ga$, we define~$\lk(A) = \cap_{v \in A} \lk(v)$. Similarly we define 
    \[ \st(v) = \lk(v) \cup \{v\},
    \]
    and~$\st(A) = \cap_{v \in A} \st(v)$.
\end{definition}

\begin{remark}
    Ferov paper: refers to length of word as the syllable length, don't confuse with word length!

    \textbf{Reduced words}: no two syllables can be joined.
\end{remark}}

\begin{definition}\label{def_simplicial_graph}
A {\it simplicial graph~$\Gamma$} is a pair~$\Gamma = (V, E)$, where~$V = V\ga$ is a non-empty set whose elements are called \emph{vertices}, and~$E = E\ga \subseteq \{\{x,y\} \mid x,y \in V\ga, x\neq y\}$ is a set of pairwise distinct vertices, whose elements are called \emph{edges}.
\end{definition}
\begin{definition}
    For a finite simplicial graph~$\Gamma = (V, E)$ we let~$A(\Gamma)$ denote the \emph{right-angled Artin group} (RAAG) based on~$\Gamma$, defined by the presentation 
$$A(\Gamma) := \langle\, V \mid uv = vu \text{ whenever } \{u,v\}\in E\,\rangle.$$
\end{definition}

\subsection{Graph products}
Here we give an overview of \emph{graph products}, and their normal form from \citep{green1990graph}.
\begin{definition}\label{def_graph_products}
Let~$\Gamma = (V,E)$ be a simplicial graph with~$V=\{1, 2, \ldots, n\}$, and let~$G_i$ be groups, indexed by~$i\in V$. The {\it graph product} of groups~$G_1, G_2, \ldots, G_n$, based on~$\Gamma$, is given by the presentation
$$G(\Gamma) \coloneqq \langle G_1, G_2, \ldots, G_n \bigm| [G_i, G_j] = 1, \forall \{i, j\} \in E \rangle.$$
One refers to the groups~$G_i$ (for~$1\leq i \leq n$) as {\it generating groups} of~$G(\Gamma)$.
\end{definition}

Taking~$G_i = \mathbb{Z}$, for all~$i \in V$, the graph product~$G(\Gamma)$ becomes the RAAG~$A(\Gamma)$; thus, RAAGs are special cases of graph products.

\begin{definition}\label{syllable_length}
Let~$G(\Gamma)$ be the graph product of the groups~$G_1, G_2, \dots, G_n$, and let~$w$ be a word in the generators of the~$G_i$.
If~$w = w_1 w_2 \dots w_r$, where each~$w_i$ is a word in the generators of only one of the generating groups, no~$w_i$ is the empty word, and~$w_i$ and~$w_{i+1}$ are not in the same generating group for all~$1 \leq i \leq r-1$, then:
\begin{itemize}[itemsep=4pt,parsep=0pt,topsep=0pt, partopsep=0pt]
    \item[(i)] the words~$w_i$ (for~$1 \leq i \leq r$) are called the \emph{syllables} of~$w$,
    \item[(ii)]~$w_1, w_2, \dots, w_r$ is called a \emph{sequence of syllables} representing the element~$w$,
    \item[(iii)] the \emph{syllable length}~$\lambda(w)$ of~$w$ is equal to~$r$, and
    \item[(iv)] the \emph{syllable length}~$\lambda(g)$ of an element~$g \in G(\Gamma)$ is the minimal syllable length of words representing~$g$, i.e.,~$\lambda(g) = \min\{\lambda(w) \mid w =_{G(\Gamma)} g\}$.
\end{itemize}
\end{definition}

The syllable length has the following basic properties:~$\lambda(1) = 0$ for the identity element~$1 \in G(\Gamma)$,~$\lambda(g) = 1$ when~$g \neq 1$ belongs to one of the generating groups, and~$\lambda(g) \geq 2$ when~$g$ does not belong to one of the generating groups.

\begin{definition}
Let~$w_1, w_2, \ldots, w_r$ be a sequence of syllables representing an element~$g$ of~$G(\Gamma)$. For~$1 \leq i < j \leq r$ we will say that the syllables~$w_i$ and~$w_j$ can {\it be joined together} if~$w_i$ and~$w_j$ belong to the same generating group, and for all~$k = i+1, \ldots, j-1$ one has~$w_iw_k =_{G(\Gamma)} w_kw_i$.
\end{definition}

In this case one can group together~$w_i$ and~$w_j$, and present~$g$ with fewer syllables.

\begin{definition}\label{def_reduced_sequence_graph_products}
A sequence~$w_1,\dots,w_r$ of syllables representing an element of~$G(\Gamma)$ is \emph{reduced} if it is the empty sequence~$\varnothing$, or if~$w_i\neq 1$ for all~$1 \leq i \leq r$ and no two syllables of the sequence can be joined.
\end{definition}

\begin{definition}\label{relation_graph_products}
Let~$\cong$ be the equivalence on reduced sequences of syllables generated by
\[
w_1,\dots,w_i,w_{i+1},\dots,w_r \ \cong\ w_1,\dots,w_{i+1},w_i,\dots,w_r
\]
whenever~$[w_i,w_{i+1}]=1$, i.e., their vertex groups are adjacent in~$\Gamma$. We refer to this relation as \emph{shuffling of syllables} or \emph{syllable shuffling}.
\end{definition}

Now we can state the normal form theorem for graph products.

\begin{theorem}[{\citealp[Theorem 3.9]{green1990graph}}]\label{normal_theorem_graph_products}
Let~$G(\Gamma)$ be a graph product of groups~$G_1, \ldots, G_n$. Each element~$g\in G(\Gamma)$ can be uniquely expressed, up to syllable shuffling, as a product:
$$g = g_1g_2\cdots g_r$$
where~$g_1, g_2, \ldots, g_r$ is a reduced sequence of syllables.
\end{theorem}

\subsection{Twisted right-angled Artin groups}\label{Twisted right-angled Artin groups}

To define T-RAAGs, we use mixed graphs, which are similar to simplicial graphs, but allow directed edges.



\begin{definition}\label{def_mixed_graph}
A \emph{mixed graph}~$\Gamma = (V, E, D, o, t)$ consists of an \emph{underlying simplicial graph}~$\overline{\Gamma} = (V, E)$, a set of \emph{directed edges}~$D \subseteq E$, and maps~$o, t \colon D \to V$ giving the \emph{origin} and \emph{terminus} of edges.  
For each~$e = \{x, y\} \in D$, we have~$o(e), t(e) \in \{x, y\}$ and~$o(e) \neq t(e)$.
\end{definition}

\begin{notation} 
Let~$\Gamma = (V, E, D, o, t)$ be a mixed graph.
\begin{itemize}[itemsep=4pt,parsep=0pt,topsep=0pt, partopsep=0pt]
    \item[(i)] If~$e = \{a,b\} \in E \setminus D$, we write~$e = [a, b]$; note also that~$e = [b, a]$.
    \item[(ii)] If~$e = \{a,b\} \in D$, we write~$e = [o(e), t(e)\rangle$.
\end{itemize}
Graphically, we present the respective edges~$[a,b]$, and~$[a,b\rangle$ as:
\begin{figure}[H]
\centering
\begin{tikzpicture}[>={Straight Barb[length=7pt,width=6pt]},thick]

\draw[fill=black] (0,0) circle (1.5pt) node[left] {$a$};
\draw[fill=black] (2,0) circle (1.5pt) node[right] {$b$};
\draw[fill=black] (5,0) circle (1.5pt) node[left] {$a$};
\draw[fill=black] (7,0) circle (1.5pt) node[right] {$b$};
\draw[thick] (0,0) -- (2,0);
\draw[thick, ->] (5,0) --  (7,0);
\end{tikzpicture}
\end{figure}
\end{notation}
In group presentations arising from graphs, the edge~$[a,b]$ denotes the commutation of~$a$ and~$b$, while the edge~$[a,b\rangle$ denotes the corresponding Klein relation~$aba=b$. By a slight abuse of notation, we also write~$[a,b] = aba^{-1}b^{-1}$ and~$[a,b\rangle = abab^{-1}$.

\begin{definition}\label{def_of_T-RAAGs}
Let~$\Gamma = (V, E)$ be a mixed graph.
Define
\[
T(\Gamma) = \Big\langle V \;\big|\; ab = ba \ \text{if } [a,b] \in E,
\quad aba = b \ \text{if } [a,b\rangle \in E
\Big\rangle.
\]
We call~$T(\Gamma)$ the \emph{twisted right-angled Artin group} based on~$\Gamma$, and~$\Gamma$ its \emph{defining graph}.  
\end{definition}

\begin{remark}\label{rem: underlying RAAG of a T-RAAG}
For a given mixed graph~$\Gamma = (V, E, D, o, t)$, we denote by~$\overline{\Gamma}$ the underlying simplicial graph. We let $A(\overline{\Gamma})$ denote the \emph{underlying RAAG} of~$T(\Gamma)$.
\end{remark}

\begin{notation}
    For a T-RAAG~$\, T(\Gamma)$ with defining graph~$\ga = (V,E)$, we let~$S = V^{\pm}$. 
\end{notation}

\begin{definition}\label{def: link and star}
    Let~$\ga$ be a mixed graph. For~$v \in V\Gamma$, we define 
    \[ \lk(v) = \{u \in V\Gamma \mid \{u,v\} \in E\Gamma \}.
    \]
    For a subset~$A \subseteq V\Gamma$, we define~$\lk(A) = \cap_{v \in A} \lk(v)$. Similarly we define 
    \[ \st(v) = \lk(v) \cup \{v\},
    \]
    and~$\st(A) = \cap_{v \in A} \st(v)$.
\end{definition}

\begin{definition}
Let~$T(\Gamma)$ be a T-RAAG with defining graph~$\Gamma$. Any element~$g \in T(\Gamma)$ can be written as a product 
\[
w = g_1 \dots g_n \in S^*
\] 
where each~$g_i = v_i^{a_i}$ for some~$v_i \in V\Gamma$ and~$a_i \in \mathbb{Z}$. Each~$g_i$ is called a \emph{syllable} of~$w$.  

For~$1 \le i < j \le n$, syllables~$g_i$ and~$g_j$ can be \emph{joined together} if~$v_i = v_j$ and for every~$k \in \{i+1, \dots, j-1\}$,~$v_k \in \mathrm{St}(v_i)$.  

A word~$w = g_1 \dots g_n$ is \emph{reduced} if it is empty or if $g_i \neq 1$ (for all $1 \leq i \leq n$) and no two syllables in~$w$ can be joined.
\end{definition}

\subsection{Normal forms and syllable shuffling in T-RAAGs}\label{sec:shuffles T-RAAGS}

\comm{
In graph products, syllable shuffling is well-defined whenever an edge exists in the defining graph between vertex groups. For T-RAAGs, we can still apply syllable shuffling, but more care is required to take into account possible reversal of powers of generators. }
In this subsection we define notation which will allow us to work with reduced words representing elements of T-RAAGs in a similar way to that of graph products, whilst also keeping track of possible changes to powers when we apply syllable shuffling. For example, in the Klein bottle group~$K = \langle a,b \mid aba=b \rangle$, the generators~$a$ and~$b$ can shuffle, but at the cost of reversing the power of~$a$, since~$ab = ba^{-1}$. We first introduce notation used similarly when working with graph products.

\begin{definition}
    Let~$g \in T(\Gamma)$, let~$w = g_{1}\dots g_{n} \in S^{\ast}$ be a reduced word representing~$g$. We define the \emph{support} of~$g$ in~$T(\Gamma)$ as
    \[ \mathrm{supp}(g) = \{v \in V\ga \mid \exists \; i \in \{1, \dots n\} \; \text{such that} \; g_{i} = v^{a} \; (a \in \Z) \}.
    \]
   \comm{ Note that this is well-defined as a result of \cref{thm:NFT}. }
\end{definition}
The following definitions match those of \citep{Ferov2016}. 
\begin{definition}
Let~$g \in T(\ga)$. Define~$\mathrm{FL}(g) \subseteq V\ga$ as the set of all~$v \in V\ga$ such that there exists a reduced word representing~$g$ which starts with the syllable~$v^{a}$, for some~$a \in \Z \setminus \{0\}$. Similarly define~$\mathrm{LL}(g) \subseteq V\ga$ as the set of all~$v \in V\ga$ such that there exists a reduced word representing~$g$ which ends with the syllable~$v^{a}$. Note~$\mathrm{FL}(g) = \mathrm{LL}(g^{-1})$. 
\end{definition}

\comm{
\begin{remark}
    FL and LL: similar to top and bottom tiles of pilings. 
\end{remark}}

\begin{remark}
    Let~$x = x_{1}\dots x_{n}$,~$y = y_{1}\dots y_{m}$ be reduced words over~$S$. Then the product~$xy$ is reduced if and only if~$\mathrm{LL}(x) \cap \mathrm{FL}(y) = \varnothing$. 
\end{remark}
Our aim is to define a function which keeps track of changing powers of generators in a word, when syllables shuffle past each other. We start by considering the generators themselves, before extending our map to words over the generating set. 
\comm{
\begin{definition}
    Let~$T = T(\Gamma) = \langle V\ga \rangle$, let~$\epsilon = \pm 1$. For each~$s \in V\ga$, define
    \begin{align*}
        \phi_{s} \colon \st(s) &\rightarrow \st(s) \\
        x^{\epsilon} &\mapsto 
        \begin{cases}
            x^{\epsilon} & sx =_{T} xs \quad \text{or} \quad sxs =_{T} x \\
            x^{-\epsilon} &  xsx =_{T} s
        \end{cases}
    \end{align*}
    For each~$s^{-1} \in V\ga^{-1}$, we define~$\phi_{s^{-1}} = \phi_{s}$.
\end{definition}}

\begin{definition}\label{Gemma notation}
Let~$\Gamma$ be a mixed graph. For each~$v \in V\Gamma$, define a map $\varphi_{v} \colon \Star(v)^{\pm 1} \rightarrow \Star(v)^{\pm 1}$, such that for $x \in \Star(v)$ and $\epsilon \in \{-1, 1\}$, we map
\begin{align*}
	x^{\epsilon} & \mapsto 
        \begin{cases}
            x^{\epsilon} & \text{if } x = v, \text{ or } [v, x] \in E\Gamma, \text{ or } [v, x \rangle \in E\Gamma \\
            x^{-\epsilon} & \text{if }  [x, v \rangle \in E\Gamma
        \end{cases}
\end{align*}
Moreover, for any~$v \in V\Gamma$, define~$\varphi_{v^{-1}} \coloneqq \varphi_{v}$.
\end{definition}
\comm{
\begin{example}
    If~$e = [a,b] \in \overline{E\Gamma}$ is an undirected edge, then~$ba =_{T} \phi_{b}(a)\phi_{a}(b) = ab$. Similarly if~$e = [a, b \rangle \in \overrightarrow{E\Gamma}$ is a directed edge, then~$ba =_{T} \phi_{b}(a)\phi_{a}(b) = a^{-1}b$. 
\end{example}
}

\begin{lemma}
If~$a, b$ are adjacent vertices in~$\Gamma$, then~$ba =_{T(\Gamma)} \varphi_{b}(a)\varphi_{a}(b)$.
\end{lemma}

\begin{proof}
If~$e = [a,b] \in E\Gamma$, then~$ba =_{T(\Gamma)} ab = \varphi_{b}(a)\varphi_{a}(b)$. Similarly, if~$e = [a, b \rangle \in E\Gamma$, then~$ba =_{T(\ga)}  a^{-1}b = \varphi_{b}(a)\varphi_{a}(b)$. Lastly, if~$e = [b,a \rangle \in E\Gamma$, then~$ba =_{T(\ga)}  ab^{-1} = \varphi_{b}(a)\varphi_{a}(b)$.
\end{proof}

\comm{
\begin{remark}
     Power of generator is irrelevant, we just need to know what type of generator we have.
\end{remark}}
The function~$\varphi_{s}$ can be extended as follows. If~$x \in \st(s) \cap \st(t)$, for some~$s,t \in V\ga$, then $\varphi_{st}(x) = \varphi_{s} \circ \varphi_{t}(x)$. This immediately implies that~$\varphi_{s^{2}}(x) = x$, which coincides with the fact that if we shuffle a generator~$s$ past~$x^{\pm 1}$ and back again, then the power of~$x$ is unchanged.

We now extend~$\varphi_{s}$ with respect to words in T-RAAGs. For notation, if~$s \in V\ga$, we let~$S_{s} = (\st(s))^{\pm}$. 

\begin{definition}\label{def:Gemma_star_notation}
    Let~$s \in V\ga$, let~$w = w_{1}\dots w_{n} \in S_{s}^{\ast}$ be a reduced word. Define
    \begin{align*}
        \varphi_{s} \colon S_{s}^{\ast} &\rightarrow S_{s}^{\ast} \\
        w &\mapsto \varphi_{s}(w_{1})\dots \varphi_{s}(w_{n}).
    \end{align*}

\end{definition}

Note that composition works in a similar way as before: suppose $s,t \in V\ga$, and let $S_{\{s, t\}} = \st(s) \cap \st(t)$. Then for any reduced word $w \in S^{\ast}_{\{s,t\}}$, we have $\varphi_{st}(w) = \varphi_{s}\circ\varphi_{t}(w)$. For reduced words $u,v \in S^{\ast}$, we say~$\varphi_{u}(v)$ and~$\varphi_{v}(u)$ are \emph{defined} if for all pairs~$(s,t)$ such that~$s \in \mathrm{supp}(u)$ and~$t \in \mathrm{supp}(v)$, one has~$\{s,t\} \in E\ga$. If this holds, then the words~$u$ and~$v$ can shuffle past each other in $T(\ga)$.

Syllable length of words and elements in T-RAAGs are defined similarly as in \autoref{syllable_length}.
In analogy with \autoref{relation_graph_products}, and using \autoref{def:Gemma_star_notation}, one has a definition of shuffling of syllables in T-RAAGs. 
\begin{definition}\label{def_traag_relation_graph_products}
Let~$\cong$ be the equivalence on reduced sequences generated by
\[
w_1,\dots,w_{i-1}, w_i,w_{i+1}, w_{i+2}, \dots, w_r \ \cong\ w_1,\dots, w_{i-1}, \varphi_{a^{m}}(w_{i+1}), \varphi_{b^{n}}(w_i), w_{i+2}, \dots,w_r
\]
whenever $w_i = a^m$, $w_{i+1} = b^n$, with $\{a, b\} \in E\Gamma$, for some $m,n \in \Z_{\neq 0}$. We refer to this relation as \emph{shuffling of syllables}. 
\end{definition}
In particular, shuffling of syllables gives~$a^m b^n  \longleftrightarrow b^n a^m$ and~$a^m b^n \longleftrightarrow b^n a^{(-1)^n m}$ for any~$m,n \in \Z$, arising from~$ab = ba$, and~$aba = b$ respectively.

\begin{theorem}\label{thm:NFT}(Normal form theorem, \citealp{foniqi2022}).
    Every element~$g \in T(\Gamma)$ can be represented by a reduced word~$g = g_{1}\dots g_{n}$, where each~$g_{i} = v_{i}^{a_{i}}$ for some~$v_{i} \in V\ga$,~$a_{i} \in \Z$. \\
    Moreover, if we have two reduced words~$x, y \in S^{\ast}$ representing~$g$, then we can obtain one word from the other via a finite sequence of syllable shuffling. In particular,~$\lambda(x) = \lambda(y)$.  
\end{theorem}
We return to our map $\phi_{s}$, and show this function is well-defined in T-RAAGs, as a consequence of the normal form theorem.

\begin{lemma}\label{lem:star_map_well_defined}
    The map~$\varphi_{s}$ is well-defined for group elements in~$T(\Gamma)$. In particular, if~$u, x, y \in S^{\ast}$ are reduced words such that~$x =_{T(\ga)} y$ and~$\varphi_{u}(x)$,~$\varphi_{u}(y)$ are defined, then 
    \[ \varphi_{u}(x) =_{T(\ga)} \varphi_{u}(y).
    \]
\end{lemma}
The following observation will be useful for our proof. Let $g \in T(\Gamma)$ and let $w = g_{1}\dots g_{n} \in S^{\ast}$ be a reduced word representing $g$. Then for any permutation $\sigma \in S_{n}$ and any word $u \in S^{\ast}$, we have 
\begin{equation}\label{eqn: old cor3.10}
    \varphi_{g_{1}\dots g_{n}}(u) = \varphi_{g_{\sigma(1)}\dots g_{\sigma(n)}}(u). 
\end{equation}
\begin{proof}[Proof of \cref{lem:star_map_well_defined}]
    By \autoref{thm:NFT}, $x$ and~$y$ are related to each other by syllable shuffling only. Therefore, it is enough to show that this result holds for words 
    \[ x = gh, \quad y = \varphi_{g}(h)\varphi_{h}(g),
    \]
    for some syllables~$g = v_{i}^{a_{i}}$,~$h = v_{j}^{a_{j}}$ such that~$\{v_{i}, v_{j}\} \in E\ga$. We have
    \begin{align*}
        \varphi_{u}(x) &= \varphi_{u}(g)\varphi_{u}(h) \\
        &=_{T(\ga)} \varphi_{g}(\varphi_{u}(h))\varphi_{h}(\varphi_{u}(g)) \\
        &= \varphi_{gu}(h)\varphi_{hu}(g) \\
        &= \varphi_{ug}(h)\varphi_{uh}(g) \\
        &= \varphi_{u}(\varphi_{g}(h))\varphi_{u}(\varphi_{h}(g)) = \varphi_{u}(y),
    \end{align*}
    where in the fourth line we use \autoref{eqn: old cor3.10}. 
\end{proof}
Using this function, we can define the following which will be useful in understanding conjugacy in T-RAAGs. 
\begin{definition}
    Let~$s \in V\ga$, let~$w = w_{1}\dots w_{n} \in S_{s}^{\ast}$ be reduced. Define a \emph{full~$\varphi_{s}$-cyclic permutation} of~$w$ to be the reduced word
    \[ w' = \varphi_{s}(w) = \varphi_{s}(w_{1}\dots w_{n}).
    \]
    \comm{
    This encodes the operation of shuffling the letter~$s^{\epsilon}$, where~$\epsilon = \pm 1$, past~$w$, i.e.
    \[ s^{\epsilon} \cdot w =_{T} \varphi_{s}(w) \varphi_{w}(s^{\epsilon}).
    \]}
\end{definition}
\begin{remark}
    Whilst this may not appear similar to a cyclic permutation of~$w$, this definition is motivated by the definition of \emph{twisted cyclic permutations} given in \citep[Section 4.2]{Crowe2024}. \comm{In particular, a full~$\varphi_{s}$ cyclic permutation of~$w$ is equivalent to a twisted cyclic permutation of~$w$, where the automorphism~$\varphi \in \mathrm{Aut}(T(\ga))$ is a composition of inversions, based on which generators switch powers as~$s$ shuffles past~$w$. }
\end{remark}
\begin{lemma}\label{cor:conj full cycles}
    Let~$w, g \in S^{\ast}$ such that~$\varphi_{w}(g),$~$\varphi_{g}(w)$ are defined. Then
    \[ \varphi_{w}(g) = g \; \Rightarrow \; w \sim \varphi_{g}(w).
    \]
\end{lemma}
\begin{proof}
    We have that
    \[ gwg^{-1} =_{T(\ga)} \varphi_{g}(w)\varphi_{w}(g)g^{-1} = \varphi_{g}(w)gg^{-1} =_{T(\ga)} \varphi_{g}(w). \qedhere
    \]
\end{proof}
Note the reverse implication does not necessarily hold. 

\begin{definition}
    Let~$u,v \in S^{\ast}$. We say~$u$ is a \emph{full conjugate preserving (FCP)~$\varphi$-cyclic permutation} of~$v$ if there exists~$s \in V\ga$ such that~$u = \varphi_{s}(v)$ and~$\varphi_{v}(s) = s$. In particular, by \autoref{cor:conj full cycles},~$u \sim v$. 
\end{definition}

\section{Conjugacy problem in T-RAAGs}\label{sec:CP in T-RAAGS}
The aim of this section is to prove \cref{thm:CP solvable}. We will adapt a conjugacy criteria for graph products given by \citep[Lemma 3.12]{Ferov2016}, taking into account the necessary operation of full conjugate preserving~$\varphi$-cyclic permutations, which can occur in T-RAAGs whilst preserving conjugacy classes. 

\subsection{Cyclic reduction}
We again establish definitions for T-RAAGs similar to that of graph products given by \citep{Ferov2016}.
\comm{
\textcolor{red}{Note from Gemma: decided to just define CR for TRAAGs, rather than also include the analogous graph product defintions, which I agree made things confusing. Hopefully this section is clearer now.}}

\begin{definition}
    Let~$T(\Gamma) = \langle S \rangle$ be a T-RAAG. Let~$w = g_{1}\dots g_{n} \in S^{\ast}$ be a reduced word representing~$g \in T(\ga)$. Call the word~$w' = g_{i+1}\dots g_{n}g_{1}\dots g_{i}$, where~$i \in \{1, \dots, n-1\}$, a \emph{cyclic permutation} of~$w$. We say the element~$g' \in T(\ga)$ is a \emph{cyclic permutation} of~$g$ if~$g'$ can be represented by a cyclic permutation of some reduced word representing~$g$. 

    We say~$w$ is \emph{cyclically reduced} if all cyclic permutations of~$w$ are reduced. We say an element~$g \in T(\ga)$ is \emph{cyclically reduced} if either~$g = Id$, or there exists at least one reduced word representing~$g$ which is cyclically reduced. 
\end{definition}
\comm{
\textcolor{red}{TODO: This remark does not look fine. Gemma: agreed, removed.}
\begin{remark}
    If~$g_{1}\dots g_{n} \in S^{\ast}$ is CR, then for all~$v \in \mathrm{LL(g)} \cap \mathrm{FL}(g)$, then~$g_{i} \neq v^{a}$ for all~$1 \leq i \leq n$.
\end{remark}}
\comm{
We now establish analogous definitions and results for T-RAAGs.

\textcolor{red}{TODO: This is a repeated definition, CR defined above. Maybe we can identify some things like this where the base group is not relevant, one just needs a generating set.} 
\begin{definition}
    Let~$w \in S^{\ast}$ be a reduced word representing an element~$g \in T(\Gamma)$. We say~$w$ is \emph{cyclically reduced} (CR) if all cyclic permutations of~$w$ are reduced.
\end{definition}}
The proof of the following result follows similarly to \cite[Lemma 3.8]{Ferov2016}, which we include here for completeness. 
\begin{lemma}
    Let~$g \in T(\Gamma)$, let~$w = g_{1}\dots g_{n} \in S^{\ast}$ be a reduced word representing~$g$. If~$w$ is cyclically reduced, then all reduced words representing~$g$ are cyclically reduced.
\end{lemma}
\comm{
Our proof follows very similarly to that of \citep{Ferov2016}, with minor edits to account for possible changes to powers after syllable shuffling.}

\begin{proof}
    Suppose~$w$ is cyclically reduced. Let~$i \in \{1, \dots n-1\}$ such that~$(g_{i}, g_{i+1})$ can shuffle, i.e.
    \[ g_{i}g_{i+1} =_{T(\ga)} \varphi_{g_{i}}(g_{i+1})\varphi_{g_{i+1}}(g_{i}).
    \]
    

    For notation we will write~$\varphi_{g_{i}} (g_{i+1}) = g'_{i+1}$ and~$\varphi_{g_{i+1}}(g_{i}) = g'_{i}$. Consider the word 
    \[ w' = g_{1}\dots g_{i-1}g'_{i+1}g'_{i}g_{i+2}\dots g_{n} =_{T(\ga)} w.
    \]
    Let~$w'' \in S^{\ast}$ be a cyclic permutation of~$w'$. We have three cases to check:
    \begin{enumerate}
        \item[$(i)$] $w'' = g_{j+1}\dots g_{i-1}g'_{i+1}g'_{i}g_{i+2}\dots g_{n}g_{1}\dots g_{j}$, where~$j < i$, 
        \item[$(ii)$] $w'' = g'_{i}g_{i+2}\dots g_{n}g_{1}\dots g_{i-1}g'_{i+1}$, and
        \item[$(iii)$] $w'' = g_{j+1}\dots g_{n}g_{1}\dots g_{i-1}g'_{i+1}g'_{i}g_{i+2}\dots g_{j}$, where~$j > i$.
    \end{enumerate}
    For~$(i)$,~$w''$ can be rewritten as 
    \[ v = g_{j+1}\dots g_{i-1}g_{i}g_{i+1}g_{i+2}\dots g_{n}g_{1}\dots g_{j},
    \]
    which is a cyclic permutation of~$w$. Therefore~$v$ is reduced by assumption, and so~$w''$ must also be reduced by \autoref{thm:NFT}. The case~$(iii)$ can be proven similarly.
 
    For~$(ii)$, we first note that by assumption the subword~$g_{i+2}\dots g_{n}g_{1}\dots g_{i-1}$ is reduced. Suppose~$w''$ is not reduced. First, suppose that the syllable~$g'_{i}$ can be joined with a syllable~$g_{k}$, for some~$k \in \{i+2, \dots, n\}$ after shuffling. If this were true, then the syllable~$g_{i}$ and~$g_{k}$ could have been joined in~$w$, which contradicts our assumption that~$w$ is reduced. Similarly suppose~$g'_{i}$ can be joined with a syllable~$g_{l}$ for some~$l \in \{1,\dots, i-1\}$. Consider the word 
    \[ p = g_{i}g_{i+1}\dots g_{n}g_{1}\dots g_{i-1}.
    \]
    Our assumption would imply that the syllables~$g_{i}$ and~$g_{l}$ can be joined in~$p$, which is a cyclic permutation of~$w$. Since $w$ is cyclically reduced, $p$ is reduced, which again gives a contradiction. A symmetric argument implies that the syllable~$g'_{i+1}$ cannot be joined with any syllable~$g_{s}$, where~$s \in \{1, \dots i-1, i+2, \dots n\}$. Also the syllables~$g'_{i}$ and~$g'_{i+1}$ cannot be joined with each other, since otherwise~$g_{i}$ and~$g_{i+1}$ could be joined in~$w$, which contradicts our assumption that~$w$ is reduced. Therefore~$w''$ must be reduced. 

    We have shown that being cyclically reduced is preserved under syllable shuffling. By \autoref{thm:NFT}, this implies that all reduced words representing~$g$ are cyclically reduced. 
    \end{proof}
\comm{
\textcolor{red}{TODO: I think this is just the result of the lemma!? do we need this extra corollary? Gemma: Yes it's needed for the overall CP algorithm, as in how do we cyclically reduced words - we should cyclically reduced in the 'obvious' way}.}
We can now state the general form for cyclically reduced elements, which will be used in the first step of our conjugacy problem algorithm later. Note this form of cyclically reduced elements matches that of RAAGs.  

\begin{cor}\label{cor:CR words}
    Let~$g \in T(\Gamma)$ be cyclically reduced. Then for any reduced word~$w \in S^{\ast}$ representing~$g$,~$w \not =_{T(\ga)} uvu^{-1}$, for any~$u,v \in S^{\ast}$ where~$l(w) > l(v)$.
\end{cor}

\subsection{Conjugacy criterion for T-RAAGs}

We can now prove the following conjugacy criteria in T-RAAGs, similar to that of \citep[Lemma 3.12]{Ferov2016}, which will allow us to prove \cref{thm:CP solvable}. 

\begin{prop}\label{prop:conj criteria}
    Let $x,y \in T(\ga)$ be cyclically reduced. Then~$x$ and $y$ are conjugate in $T(\ga)$ if and only if there exists a finite sequence~$w = w_{0}, w_{1}, \dots,w_{m} = w'$ of cyclically reduced words, such that $w$ and $w'$ represent the elements $x$ and $y$ respectively, and each pair $(w_{i}, w_{i+1})$, for $0 \leq i \leq m-1$, is related via:
    \begin{itemize}
        \item[$(i)$] syllable shuffling,
        \item[$(ii)$] a cyclic permutation, or
        \item[$(iii)$] a full conjugate preserving~$\varphi$-cyclic permutation.
    \end{itemize}
In particular, if~$x$  and $y$ are conjugate, then~$|x| = |y|$, $\lambda(x) = \lambda(y)$ and~$\mathrm{supp}(x) = \mathrm{supp}(y)$.
\end{prop}
Note if~$T(\Gamma) = A(\Gamma)$ is a RAAG, then FCP~$\varphi$-cyclic permutations do not have any affect on words, i.e.~$\varphi_{s}(w) = w$. In particular, this result coincides with the well known fact in RAAGs that two cyclically reduced elements are conjugate if and only if they are related by a finite sequence of commutations (i.e. syllable shuffles) and cyclic permutations. 

The following notation will be required in the proof of \autoref{prop:conj criteria}.
\begin{definition}[P-S decomposition]
Let~$g \in T(\Gamma)$. Define~$S(g) = \supp(g)  \cap \st(\supp(g))$, and~$P(g) = \supp(g) \setminus S(g)$. One can write~$g$ uniquely as a reduced product~$g = s(g)p(g)$, with~$\supp(s(g)) = S(g)$ and~$\supp(p(g)) = P(g)$.
\end{definition}

One has~$\fl(g)  = S(g) \sqcup \fl(p(g))$,~$\llt(g)  = S(g) \sqcup \llt(p(g))$, and~$S(p(g)) = \emptyset$.

\begin{remark}
    Let~$g \in T(\Gamma)$ where~$g_{1}\dots g_{n}$ is a reduced expression representing~$g$. If~$v \in S(g)$, then there exists~$g_{i} \in  \langle v \rangle$ such that~$g_{i}$ shuffles with all other syllables in~$g$. In particular,~$v \in \mathrm{FL}(g) \cap \mathrm{LL}(g)$. 
\end{remark}

\begin{proof}[Proof of \autoref{prop:conj criteria}]
    Each of the three operations preserve conjugacy in~$T(\Gamma)$, so the reverse direction is immediate. For the forward direction, WLOG assume that~$\lambda(x) \geq \lambda(y)$. Let~$X \subseteq T(\ga)$ denote the set of all elements obtained from~$x$ via cyclic and FCP~$\varphi$-cyclic permutations. In particular,~$X^{T_{S(x)}} \subseteq x^{T(\Gamma)} = \{gxg^{-1} \mid g \in T(\ga)\}$. Choose~$x' \in X^{T_{S(x)}}$ such that the corresponding~$g' \in T(\Gamma)$, where~$g'x'g'^{-1} = y$, is of minimal syllable length. 

    \textbf{Claim:} There exists~$x'' \in X^{T_{S(x)}}$,~$g'' \in T(\Gamma)$ such that~$\lambda(g'') = \lambda(g')$,~$g''x''g''^{-1} = y$, and~$g''x''$ is reduced.\\
    \textbf{Proof of claim:} We will prove by induction on~$c = |\mathrm{LL}(g') \cap \mathrm{FL}(x') |$. 

    Base case:~$c = 0$. Here~$g'x'$ is reduced, so we can set~$g'' = g'$ and~$x'' = x'$ for the claim to hold.

    Now suppose~$c > 0$ and the claim holds for all~$c' < c$. Let~$g_{1}\dots g_{k}$ and~$x_{1}\dots x_{n}$ be reduced expressions for~$g'$ and~$x'$ respectively. WLOG let~$v \in \mathrm{LL}(g') \cap \mathrm{FL}(x')$ such that~$g_{k}, x_{1} \in \langle v \rangle$. We have    \[ y =_{T(\ga)} g_{1}\dots (g_{k}x_{1})x_{2}\dots x_{n}\cdot g^{-1}_{k}\dots g^{-1}_{1}.
    \]
    Suppose~$g_{k} = x^{-1}_{1}$. Then we would have
    \[ y =_{T(\ga)} g_{1}\dots g_{k-1}x_{2}\dots x_{n}x_{1}g^{-1}_{k-1}\dots g^{-1}_{1}.
    \]
    One could then replace~$x'$ by~$x_{2}\dots x_{n}x_{1}$, which is a cyclic permutation of~$x'$, and also can replace~$g'$ by~$g_{1}\dots g_{k-1}$. Since~$x_{2}\dots x_{n}x_{1} \in X^{T_{S(x)}}$, this contradicts the minimality of~$g'$. Therefore~$g_{k} \neq x^{-1}_{1}$. 

    \comm{ OLD ARGUMENT GEMMA
    We now consider what happens if~$x_{1}$ can shuffle to the end of~$x'$. First, suppose 
    \[ x_{1}\cdot x_{2}\dots x_{n} =_{T} \varphi_{v}(x_{2}\dots x_{n})x^{-1}_{1},
    \]
    i.e.~$\varphi_{x_{2}\dots x_{n}}(x_{1}) = x^{-1}_{1}$. Then we have
    \[ y =_{T} g_{1}\dots g_{k} \varphi_{v}(x_{2}\dots x_{n})x^{-2}_{1}g^{-1}_{k-1}\dots g^{-1}_{1}.
    \]
    
    \textcolor{red}{Not 100\% convinced about this step - possibility of torsion}

    If cancellation occurs, then~$x^{2}_{1} = 1$, i.e.~$x_{1} = x^{-1}_{1}$. Then we would have
    \[ y _{T}  g_{1}\dots g_{k-1} x_{1}\varphi_{v}(x_{2}\dots x_{n})g^{-1}_{k-1}\dots g^{-1}_{1}.
    \]
    Note 
    \[ x' = x_{1}\dots x_{n} =_{T} \varphi_{v}(x_{2}\dots x_{n})x^{-1}_{1} =_{T} \varphi_{v}(x_{2}\dots x_{n})x_{1} \leftrightarrow x_{1}\varphi_{v}(x_{2}\dots x_{n}),
    \]
    and so we could have chosen~$x' = x_{1}\varphi_{v}(x_{2}\dots x_{n})$ which breaks our assumption of minimality of~$g'$.  

    Secondly, suppose~$\varphi_{x_{2}\dots x_{n}}(x_{1}) = x_{1}$, i.e.
    \[ x_{1}\cdot x_{2}\dots x_{n} =_{T} \varphi_{v}(x_{2}\dots x_{n})x_{1}.
    \]
    Then again we break our minimality assumption. In particular,~$v \not \in \mathrm{LL}(g'x_{1}) \cap \mathrm{FL}(x_{2}\dots x_{n})$. This implies that
    \[ \mathrm{LL}(g'x_{1}) \cap \mathrm{FL}(x_{2}\dots x_{n}x_{1}) \subsetneq \mathrm{LL}(g') \cap \mathrm{FL}(x').
    \]
    From here we apply the inductive hypothesis, which proves the claim.}

    If~$v \in S(x)$, then~$g_kxg_k^{-1} \in X^{T_{S(x)}}$ which again contradicts the choice of~$x'$ and~$g'$. Hence,~$v \not \in S(x)$, and therefore~$\llt(g') = \llt(g'x_1)$ and also~$v \not \in \fl(x_2 \cdots x_nx_1)$. Note that if~$g_i$ can be shuffled to the end of~$g'$, then~$\{u, v\} \in E\Gamma$, where~$g_i \in \langle u \rangle$. If~$w \in \fl(x_2 \cdots x_nx_1) \setminus \fl(x')$ then~$\{v, w\} \not \in E\Gamma$ and so~$w \not \in \llt(g'x_1)$.

From this we can conclude that~$v \not \in \llt(g'x_1) \cap \fl(x_2 \cdots x_nx_1) \subseteq \llt(g') \cap \fl(x')$, hence the intersection~$\llt(g'x_1) \cap \fl(x_2 \cdots x_nx_1)$ is a proper subset of~$\llt(g') \cap \fl(x')$. The inductive hypothesis completes the proof of the claim. 
    
    We now have that~$g''x'' = yg''$. Since~$g''x''$ is reduced,~$\lambda(g''x'') = \lambda(g'') + \lambda(x'') = \lambda(g') + \lambda(x') = n+k$. Also~$\lambda(yg'') \leq \lambda(y) + \lambda(g'') = m+k$, where~$\lambda(y) = m$. Therefore~$\lambda(x) \leq \lambda(y)$. By our assumption, this implies~$\lambda(x) = \lambda(y)$, and so~$yg''$ is also reduced. We now assume~$y_{1}\dots y_{n}$,~$x_{1}\dots x_{n}$ and~$g_{1}\dots g_{k}$ are reduced representatives of~$y, x''$ and~$g''$ respectively. We have
    \[ y_{1}\dots y_{n} =_{T(\ga)} g_{1}\dots g_{k}x_{1}\dots x_{n}g^{-1}_{k}\dots g^{-1}_{1}.
    \]
    By \autoref{thm:NFT}, there exists reductions on the right hand side after shuffling. Assume that the syllable~$g_k^{-1}$ can be joined up with~$g_j$ for some~$1 \leq j \leq k$. This means that~$g_k$ shuffles with~$x_i$ for all~$1 \leq i \leq n$. In particular,
    \[ y_{1}\dots y_{n} =_{T(\ga)} g_{1}\dots g_{k-1}g^{2}_{k}x_{1}\dots x_{n}g^{-1}_{k-1}\dots g_{1}^{-1}.
    \]
    Rearranging gives
\[  y_1 \cdots y_n g_1 \cdots g_{k-1} =_{T(\ga)} g_1 \cdots g_{k-1}g_k^2 x_1 \cdots x_n
\]
The right hand side is reduced since~$g''x''$ is reduced, and has syllable length~$n+k$. The left hand side is also reduced, since it is a subword of~$yg''$. However the syllable length here is~$n+k-1$, which contradicts \cref{thm:NFT}. In particular,~$g_{k} \not \in S(x'')$.
    
\comm{
    First suppose~$g_{k} \in \lk(\mathrm{supp}(x'))$. If~$g_{k}x' = \varphi_{v}(x')g_{k}$ then, similar to before, cancellation contradicts our assumption that~$g''$ is of minimal length. Similarly we contradict minimality if~$g_{k}x' = \varphi_{v}(x')g^{-1}_{k}$ and \myRed{~$g^{2}_{k} = 1$}. }
    
    WLOG we assume there exists~$w \in \mathrm{LL}(x'') \cap \mathrm{LL}(g'')$ such that~$g_{k}, x_{n} = w^{\pm a}$, for some $a \in \Z_{\neq 0}$. If~$x_{n} = g^{-1}_{k}$, we have
    \[ y_{1}\dots y_{n}g_{1}\dots g_{k-1} =_{T(\ga)} g_{1}\dots g_{k} x_{1}\dots x_{n-1}x^{2}_{n}.
    \]
    Both sides are reduced by construction, but differ in syllable length, which contradicts \autoref{thm:NFT}. Therefore~$x_{n} = g_{k}$ and so after cancellation we have
    \[ y_{1}\dots y_{n} =_{T(\ga)} g_{1}\dots g_{k-1} x_{n}x_{1}\dots x_{n-1}g^{-1}_{k-1}\dots g^{-1}_{1}. 
    \]
    Again a contradiction, as we could replace~$x''$ with the cyclic permutation~$x_{n}x_{1}\dots x_{n-1}$, and replace~$g''$ with a shorter conjugator. Indeed this contradiction always occurs unless~$g=1$. Therefore~$y = x'' \in X$, which completes the proof. 
\end{proof}

\begin{theorem}\label{thm:CP solvable}
    The conjugacy problem in T-RAAGs is solvable.
\end{theorem}

\begin{proof}
    Our algorithm is as follows:

    \textbf{Input:} T-RAAG~$T(\ga)$, reduced words~$u,v \in S^{\ast}$. 

    \textbf{Step 1: Cyclic reduction}
    \begin{adjustwidth}{1.5cm}{}
        Cyclically reduce~$u$ and $v$, using \autoref{cor:CR words}.
    \end{adjustwidth}
    \textbf{Step 2: Set of representatives}
    \begin{adjustwidth}{1.5cm}{}
        Let~$D_{u}$ denote the set of all words obtained from~$u$ via syllable shuffles, cyclic permutations and FCP~$\varphi$-cyclic permutations. By \autoref{prop:conj criteria},~$D_{u}$ is finite, and it remains to check whether~$v \in D_{u}$. If~$v \in D_{u}$, then \textbf{Output} = \texttt{\color{blue}True}. Otherwise, \textbf{Output} = \texttt{\color{blue}False}.
    \end{adjustwidth}
    \vspace{-15pt}
\end{proof}
The following question remains open.
\begin{question}
    What is the time complexity of the conjugacy problem in T-RAAGs?
\end{question}
\begin{remark}
    The linear-time algorithm for RAAGs given by \citep{CGW} does not seem to extend to T-RAAGs. The machinery they use, namely \emph{pilings}, does not have an analogous definition for T-RAAGs, since we lose uniqueness of piling representatives. Moreover, their algorithm relies on the fact that two words in a cyclic normal form are conjugate in a RAAG if and only if they are equal up to a cyclic permutation. For T-RAAGs, this is not necessarily true. For example, in the Klein-bottle group~$K = \langle a, b \mid aba = b \rangle$, the words~$a$ and~$a^{-1}$ represent conjugate elements, and are cyclic normal forms, but are not equal up to cyclic permutation.  
\end{remark}

\subsection{Linear-time example}
We conclude this section by considering a subclass of T-RAAGs with a specific defining graph, of which we can determine the complexity of the conjugacy problem for the associated T-RAAG. This is due to a connection with another decision problem, known as the \emph{twisted conjugacy problem.}

\begin{definition}
    Let $G = \langle X \rangle$, let $u,v \in X^{\ast}$, and let $\varphi \in \mathrm{Aut}(G)$. We say $u$ and $v$ are \emph{$\varphi$-twisted conjugate} if there exists $w \in G$ such that $v =_{G} \varphi(w)^{-1}uw$. 

    The \emph{$\varphi$-twisted conjugacy problem} for $G$, denoted $\mathrm{TCP}_{\varphi}(G)$, takes as input two words $u,v \in X^{\ast}$, and decides whether these represent group elements which are $\varphi$-twisted conjugate in $G$. 
\end{definition}
Consider a T-RAAG $T(\ga)$ with vertices~$V\ga = \{x_{1}, \dots, x_{n}, y \}$, such that the only directed edges in~$\Gamma$ are those of the form~$[x_{i}, y\rangle$, for some~$i \in \{1, \dots, n\}$. In this case,~$T(\ga)$ is isomorphic to a cyclic extension of the RAAG~$A(\widehat{\ga})$, where~$\widehat{\ga}$ is the induced subgraph of~$\ga$ with vertices~$\{x_{1}, \dots, x_{n}\}$. In this scenario, the conjugacy problem in $T(\ga)$ can be solved in linear time. 
\begin{theorem}\label{thm:subcase result}
    Let~$T(\Gamma)$ be a T-RAAG such that 
    \[ T(\Gamma) \cong A(\widehat{\Gamma}) \rtimes_{\varphi} \Z = \langle V\widehat{\Gamma}, t \mid R(A(\widehat{\Gamma})), \; txt^{-1} = \varphi(x) \; (x \in V\widehat{\Gamma}) \rangle,
    \]
    where~$\varphi \in \mathrm{Aut}(A(\widehat{\Gamma}))$ is a composition of inversions. Then the conjugacy problem in~$T(\Gamma)$ can be solved in linear time. 
\end{theorem}

\begin{proof}
    By \citep{Orbits} the conjugacy problem in~$T(\Gamma)$ is equivalent to the twisted conjugacy problem~$\mathrm{TCP}_{\varphi}(A(\widehat{\Gamma}))$, with respect to inversions. This is solvable in linear time by \citep[Theorem 4.1]{Crowe2024}.
\end{proof}

\section{Biautomaticity of T-RAAGs}\label{sec:biauto}
For our purposes, we will work with the following definition of biautomaticity given by Sarah Rees in \citep[Chapter 14]{Thurston2}. Recall, for $G = \langle X \rangle$, for~$k \in \mathbb{N}$ and words~$w,v \in (X^{\pm})^{\ast}$, we say~$w$ and~$v$ \emph{$k$-fellow travel} in~$G$, if for each~$i \leq \mathrm{max}\{\ell(w), \ell(v)\}$, one has~$|\mathrm{pre}_{i}(w)^{-1}\mathrm{pre}_{i}(v)| \leq k$ in~$\mathsf{Cay}(G,X)$. Equivalently we say the paths~$(1,w)$ and~$(1,v)$ of~$\mathsf{Cay}(G,X)$~$k$-fellow travel. 

\begin{definition}\label{defn:automatic}
    Let~$G = \langle X \rangle$. We define a \emph{language} for~$G$ over~$X^{\pm}$ to be a subset of~$(X^{\pm})^{\ast}$ that contains at least one representative for each element of~$G$. We say~$G$ is \emph{automatic} if 
    \begin{enumerate}
        \item[(A1)] there exists a regular language for~$G$ over~$X^{\pm}$, and 
        \item[(A2)] there exists~$k \in \Z$ such that, for each~$y \in X^{\pm} \cup \{1\}$, and for any~$w,v \in L$ with~$wy =_{G} v$, the paths~$(1, w), (1,v)$~$k$-fellow travel in~$\mathsf{Cay}(G,X)$. 
    \end{enumerate}
    We say~$G$ is \emph{biautomatic} if it is automatic and also satisfies the following:
    \begin{enumerate}
        \item[(A3)] there exists~$k \in \Z$ such that, for each~$y \in X^{\pm} \cup \{1\}$, and for any~$w,v \in L$ with~$yw =_{G} v$, the paths~$(1, w), (1,v)$~$k$-fellow travel in~$\mathsf{Cay}(G,X)$. 
    \end{enumerate}
\end{definition}
Biautomatic groups have solvable conjugacy problem \citep{Gersten1997}, however it is unknown whether automatic groups have solvable conjugacy problem. RAAGs were first shown to be biautomatic by \citep{VanWyk1994}. This was later shown by \citep{Hermiller1995} in the more general context of graph products. We take inspiration primarily from \citep{Hermiller1995}, to describe another method that implies biautomaticity in RAAGs, which we will then apply to twisted RAAGs. 

In \citep{Hermiller1995}, a finite state automaton is constructed which accepts a normal form for a RAAG~$A(\Gamma)$ over the monoid generating set~$V\Gamma$. We will adapt this automaton, which will then accept the normal form given by \citep{foniqi2024twisted} for T-RAAGs over~$V\Gamma$, and so T-RAAGs satisfy (A1) from \cref{defn:automatic}. We will then show that (A2) and (A3) both hold in T-RAAGs, with respect to the language accepted by our new automaton.

Set~$V\Gamma = \{v_{1}, \dots, v_{m}\}$. We define an order on the generating set~$V\Gamma^{\pm}$ as 
\[ v_{1} < v_{1}^{-1} < v_{2} < v_{2}^{-1} < \dots < v_{m} < v_{m}^{-1}.
\]
Our proposed automaton $\mathcal{M}_{\Gamma}$ is constructed as follows. 
\begin{enumerate}
    \item Fix a start state~$Q$, and construct~$2m$ accept states, which we label by the generators $V_{i}, V^{-1}_{i} \in V\ga^{\pm}$, for~$i = 1, \dots, n$. Add edges from $Q$ labelled $v_{i}$ (resp. $v^{-1}_{i}$) to each state $V_{i}$ (resp. $V^{-1}_{i}$), for all $1 \leq i \leq n$.    
    \item Add a loop at each state $V_{i}$ (resp. $V^{-1}_{i}$) labelled $v_{i}$ (resp. $v^{-1}_{i}$). Fix a fail state $F$, and add an edge from each state $V_{i}$ (resp. $V^{-1}_{i})$, labelled $v^{-1}_{i}$ (resp. $v_{i}$), to $F$. 
    \item For each $i,j \in \{1, \dots, m\}$, $i \neq j$, suppose that either $\{v_{i}, v_{j}\} \not \in E\Gamma$, or $\{v_{i}, v_{j}\} \in E\Gamma$ and $v^{\pm 1}_{i} < v^{\pm 1}_{j}$. Then we add edges from the state $V_{i}$ to the states $V_{j}$ and $V^{-1}_{j}$, labelled $v_{j}$ and $v^{-1}_{j}$ respectively. Similarly, we add edges from the state $V^{-1}_{i}$ to the states $V_{j}$ and $V^{-1}_{j}$, labelled $v_{j}$ and $v^{-1}_{j}$ respectively. If neither of these conditions hold, that is, $\{v_{i}, v_{j}\} \in E\Gamma$ and $v^{\pm 1}_{i} > v^{\pm 1}_{j}$, then we add edges from the state $V_{i}$ to $F$, labelled $v_{j}$ and $v^{-1}_{j}$, and similarly add edges from the state $V^{-1}_{i}$ to $F$, labelled $v_{j}$ and $v^{-1}_{j}$.
\end{enumerate}

We let~$\mathcal{M}_{\Gamma}$ denote the automaton built from this construction. By \cref{thm:NFT}, there exists a unique shortlex shortest word, with respect to the order on~$V\Gamma^{\pm}$ given above, representing each group element of~$T(\Gamma)$. To determine whether a word~$w = w_{1}\dots w_{n} \in V\Gamma^{\pm}$, where $w_{i} = v^{\ve}_{j_{i}}$ for some $i \in \{1, \dots, n\}$, $j_{i} \in \{1, \dots, m\}$, $\ve = \pm 1$, is the unique shortlex shortest representative for a group element in~$T(\Gamma)$, we read letters from left to right, and consider the following options when we read a letter~$w_{i}$, for~$1 \leq i \leq n$:

\begin{itemize}
    \item[$(i)$] We read the first letter $w_{1}$, which is accepted by $\mathcal{M}_{\Gamma}$; this corresponds to Step 1 of our construction above.  
    \item[$(ii)$] If~$w_{i} = w_{i-1}$, this letter is accepted and we read the next letter of~$w$. If~$w_{i} = w_{i-1}^{-1}$, then~$w$ cannot be a normal form, since~$w$ is not geodesic, and we reject $w$. In $\mathcal{M}_{\ga}$, this corresponds to Step 2 of our construction. 
    \item[$(iii)$] If either $\{v_{j_{i-1}}, v_{j_{i}}\} \not \in E\ga$, or $\{v_{j_{i-1}}, v_{j_{i}}\} \in E\ga$ and $v^{\pm 1}_{j_{i-1}} < v^{\pm 1}_{j_{i}}$, then we accept the subword $w_{1}\dots w_{i}$ and read the next letter of $w$. Otherwise, we have that $\{v_{j_{i-1}}, v_{j_{i}}\} \in E\ga$ and $v^{\pm 1}_{j_{i-1}} > v^{\pm 1}_{j_{i}}$, in which case the subword $w_{1}\dots w_{i-1}w_{i}$ is no longer shortlex shortest, since it can be rewritten as~$w_{1}\dots \varphi_{i-1}(w_{i})\varphi_{i}(w_{i-1})$ in~$T(\Gamma)$. In this case we reject $w$. All three scenarios correspond to Step 3 of our construction of $\mathcal{M}_{\ga}$. 
\end{itemize}
The following result is then immediate from the discussion above. 

\begin{cor}
    Let~$L$ be the language accepted by the automaton~$\mathcal{M}_{\Gamma}$. Then~$L$ is precisely the language of normal forms for~$T(\Gamma)$ over~$V\Gamma$. 
\end{cor} 
In particular, T-RAAGs satisfy (A1). To show (A2) and (A3), we first prove these conditions for RAAGs, with respect to the language accepted by our automaton~$\mathcal{M}_{\Gamma}$.

\begin{lemma}\label{prop:RAAGs A2 A3}
    RAAGs satisfy (A2) and (A3) with respect to the language~$L$ accepted by the automaton~$\mathcal{M}_{\Gamma}$. 
\end{lemma}

\begin{proof}
    Let~$w,v \in L$ such that~$wy =_{G} v$ for some~$y \in V\Gamma^{\pm}$. We can write~$w = w_{1}y^{n}w_{2}$ for some~$n \in \Z$, such that for every syllable~$g_{s}$ in~$w_{2}$, where~$g_{s} \in \langle v_t \rangle$ for some~$t \in \{1, \dots, n\}$, we have that~$(y, v_{t}) \in E\Gamma$, and~$y < v^{\pm 1}_{t}$ with respect to the order on the generating set. 

    Hence we can write~$v_{2} =_{G} w_{1}y^{n+1}w_{2}$, such that~$w_{1}y^{n+1}w_{2} \in L$. It is then immediate that~$w$ and~$v$ 2-fellow travel, and so (A2) holds for RAAGs. (A3) also holds using a similar proof.
\end{proof}
The following result will allow us to complete our proof of biautomaticity in T-RAAGs. 

\begin{theorem}\citep[Theorem 5.9]{foniqi2022}
    Let~$T(\Gamma)$ be a twisted RAAG, and let~$A(\Gamma)$ be its underlying RAAG. Then the Cayley graphs of~$T(\Gamma)$ and~$A(\Gamma)$ are isomorphic as undirected graphs, with respect to the standard generating set~$S = V\Gamma^{\pm}$. 
\end{theorem}
In particular, distances are preserved between~$\mathsf{Cay}(A(\Gamma), S)$ and~$\mathsf{Cay}(T(\Gamma), S)$, and so T-RAAGs also satisfy (A2) and (A3), using \Cref{prop:RAAGs A2 A3}. The following is then immediate, which gives an alternative positive solution to the conjugacy problem in T-RAAGs.  

\begin{theorem}\label{thm:biauto}
    T-RAAGs are biautomatic.
\end{theorem}

\comm{
\subsection{Potentially simpler solution}
\begin{enumerate}
    \item Construct automaton which accepts normal form for RAAGs~$A(\Gamma)$ with respect to~$V(\Gamma)^{\pm}$.
    \item Show the language it accepts~$k$-fellow travel for both right and left multiplication of generators. Note these two facts give us (again) biautomaticity in RAAGs.
    \item Apply to T-RAAGs:
    \begin{enumerate}
        \item Show the automaton constructed before also accepts normal form for~$T(\Gamma)$
        \item Since Cayley graph Q.I., this should imply~$k$-fellow traveller for~$T(\Gamma)$ as well. Even stronger statement: use isomorphism (as undirected graphs) between Cayley graph of~$A(\Gamma)$ and~$T(\Gamma)$. 
    \end{enumerate}
\end{enumerate}}

\noindent{\textbf{{Acknowledgments.}}} 
Islam Foniqi acknowledges support from the EPSRC Fellowship grant EP/V032003/1 ‘Algorithmic, topological and geometric aspects of infinite groups, monoids and inverse semigroups’.
\newpage
\noindent\textit{\\ Gemma Crowe,\\
University of Manchester, and the Heilbronn Institute for 
Mathematical Research, \\
Bristol, UK\\ }
{Email: gemma.crowe@manchester.ac.uk}

\noindent\textit{\\ Islam Foniqi,\\
The University of East Anglia\\ 
Norwich (United Kingdom)\\}
{Email: i.foniqi@uea.ac.uk}

\bibliography{main}

\end{document}